\documentclass[12pt]{article}
\usepackage{amsmath,amsfonts,amssymb,amsthm,fullpage}
\usepackage{stmaryrd}
\usepackage[hidelinks]{hyperref}
\usepackage{tikz,subcaption}
\usetikzlibrary{backgrounds}
\newtheorem{thm}{Theorem}
\newtheorem{lem}{Lemma}

\newtheorem{prop}{Proposition}
\newtheorem{conj}{Conjecture}

\theoremstyle{definition}
\newtheorem{Def}{Definition}
\newtheorem*{rem}{Remark}
\DeclareRobustCommand\longtwoheadrightarrow
     {\relbar\joinrel\twoheadrightarrow}

\newcommand\rarrow[1][r]{\overset{#1}{\longrightarrow}}
\newcommand\homrarrow[1][r]{{\overset{#1}{\longtwoheadrightarrow}}}
\newcommand\nrarrow[1][r]{\overset{#1}{\longarrownot\longrightarrow}}

\newcommand\B{\mathrm{B}}
\newcommand\K{\mathcal{K}}
\newcommand\M{\mathcal{M}}

\tikzset{blowup/.style={draw, fill=white, circle, inner sep=0pt, minimum size=21pt}}
\tikzset{vert/.style={draw, fill=black, circle, inner sep=2pt}}

\title{Extremal and Ramsey results on graph blowups}
\author{Jacob Fox\thanks{Department of Mathematics, Stanford University, Stanford, CA 94305, USA. Email: \url{jacobfox@stanford.edu}. Research supported by a Packard Fellowship and by NSF award DMS-1855635.} \and Sammy Luo\thanks{Department of Mathematics, Stanford University, Stanford, CA 94305, USA. Email: \url{sammyluo@stanford.edu}. Research supported by NSF GRFP Grant DGE-1656518.} \and Yuval Wigderson\thanks{Department of Mathematics, Stanford University, Stanford, CA 94305, USA. Email: \url{yuvalwig@stanford.edu}. Research supported by NSF GRFP Grant DGE-1656518.}}

\begin{document}
\maketitle
\begin{abstract}
	Recently, Souza introduced blowup Ramsey numbers as a generalization of bipartite Ramsey numbers. For graphs $G$ and $H$, say $G\rarrow H$ if every $r$-edge-coloring of $G$ contains a monochromatic copy of $H$. Let $H[t]$ denote the $t$-blowup of $H$. Then the blowup Ramsey number of $G,H,r,$ and $t$ is defined as the minimum $n$ such that $G[n] \rarrow H[t]$. Souza proved upper and lower bounds on $n$ that are exponential in $t$, and conjectured that the exponential constant does not depend on $G$. We prove that the dependence on $G$ in the exponential constant is indeed unnecessary, but conjecture that some dependence on $G$ is unavoidable.

	An important step in both Souza's proof and ours is a theorem of Nikiforov, which says that if a graph contains a constant fraction of the possible copies of $H$, then it contains a blowup of $H$ of logarithmic size. We also provide a new proof of this theorem with a better quantitative dependence.
\end{abstract}
\section{Introduction}
A graph $G$ is called \emph{$r$-Ramsey} for a graph $H$, denoted $G\rarrow H$, if every $r$-edge-coloring of $G$ contains a monochromatic copy of $H$. Given a graph $H$ and an integer $t$, the \emph{$t$-blowup} of $H$, denoted $H[t]$, is the graph obtained from $H$ by replacing every vertex of $H$ by an independent set of order $t$, and replacing every edge of $H$ by a complete bipartite graph $K_{t,t}$ between the corresponding parts. Say that a copy of $H[t]$ in $G[n]$ is \emph{canonical} if it is the $t$-blowup of a copy of $H$ in $G$. Recently, Souza \cite{Souza} introduced the notion of \emph{blowup Ramsey numbers}, which are a natural generalization of several well-studied problems in Ramsey theory, such as that of bipartite Ramsey numbers.
\begin{Def}[Souza \cite{Souza}]
	Let $G,H$ be graphs and $r$ an integer such that $G \rarrow H$. For an integer $t$, define the \emph{blowup Ramsey number} $\B(G \rarrow H;t)$ to be the minimum $n$ such that every $r$-coloring of $G[n]$ contains a monochromatic canonical copy of $H[t]$. 
\end{Def}
Souza proved that these numbers exist and are finite, and further obtained an exponential upper bound on them.
\begin{thm}[Souza \cite{Souza}]\label{thm:souza-ub}
	Let $G,H$ be graphs and $r$ an integer such that $G \rarrow H$. Then there is a number $c=c(G,H,r)$ such that for every $t$,
	\[
		\B(G \rarrow H;t) \leq c^t.
	\]
\end{thm}
Moreover, using the Lov\'asz Local Lemma, Souza showed that an exponential-type bound is necessary. Indeed, he proved that if $t$ is sufficiently large in terms of $G$ and $n \leq (c')^t$ for some constant $c'=c'(H,r)>1$, then there exists an $r$-edge-coloring of $G[n]$ with no monochromatic canonical copy of $H[t]$.

The exponential constant in Souza's upper bound depends on $G$, while the exponential constant in his lower bound does not depend on $G$. Souza conjectured that the dependence on $G$ in Theorem \ref{thm:souza-ub} is unnecessary. In this paper, our main result is that the exponential constant indeed does not depend on $G$, but our upper bound nevertheless has some dependence on $G$. More precisely, we prove the following result. 
\begin{thm}\label{thm:no-g-ub}
	Let $G,H$ be graphs and $r\geq 2$ an integer such that $G \rarrow H$. There exist constants $a=a(G,H,r)$ and $b=b(H,r)$ such that for every integer $t$,
	\[
		\B(G \rarrow H;t) \leq a \cdot b^t.
	\]
	Moreover, for $\gamma>0$ sufficiently small with respect to $r$, we may take $b=r^{(r+\gamma)^{|E(H)|-1}}$, so long as $a$ is sufficiently large with respect to $\gamma$.
\end{thm}
This result shows that if we are only interested in the exponential rate of growth of $\B(G \rarrow H;t)$ as a function of $t$, then indeed the choice of $G$ does not matter. However, for fixed $t$, the upper bound in Theorem \ref{thm:no-g-ub} does depend on $G$, and we believe that this dependence is in fact necessary for some $H$; for more details, see the concluding remarks. 

An important step in Souza's proof of Theorem \ref{thm:souza-ub} is the following result of Nikiforov, which says that a graph with many copies of $H$ must contain a blowup of $H$ of logarithmic size.
\begin{thm}[Nikiforov \cite{Nikiforov1,Nikiforov2}]\label{thm:nikiforov-original}
	For every $\eta>0$ and every graph $H$ on $k$ vertices, there exists a constant $\lambda>0$ such that the following holds for all $n$ sufficiently large. Let $G$ be a graph on $n$ vertices containing at least $\eta n^k$ copies of $H$. Then $G$ contains a blowup $H[t]$, where $t=\lambda \log n$. Moreover, one can take $\lambda=\eta^k$ if $H$ is a clique and $\lambda=\eta^{k^2}$ if $H$ is an arbitrary graph.
\end{thm}
As a consequence of our main technical result, whose proof is inspired by Nikiforov's original proof but further adds ideas from graph regularity, we provide a new proof of Theorem \ref{thm:nikiforov-original} with a better quantitative dependence between $\lambda$ and $\eta$. Specifically, we prove that one can take $\lambda=\eta^{1-1/|E(H)|+o(1)}$ in Theorem \ref{thm:nikiforov-original}; see Section \ref{sec:nikiforov} for details.

The paper is organized as follows. In Section \ref{sec:lemmas}, we state and prove several technical lemmas, related to regularity of graphs, that we will need in the proof of Theorem \ref{thm:no-g-ub}. In Section \ref{sec:nikiforov}, we use these lemmas to state and prove our stronger version of Theorem \ref{thm:nikiforov-original}. In Section \ref{sec:upper-bound}, we again use these lemmas to prove Theorem \ref{thm:no-g-ub}. We end with some concluding remarks. For the sake of clarity of presentation, we systematically omit floor and ceiling signs whenever they are not crucial. All logarithms in this paper are base $e$ unless otherwise stated.

\section{Tools from regularity theory}\label{sec:lemmas}
Our first technical result is the weak regularity lemma of Duke, Lefmann, and R\"odl \cite{DuLeRo}. In fact, we will use a generalization of it due to Fox and Li \cite{FoLi} which is well-adapted for dealing with colorings, as opposed to single graphs. The main advantage of their result over that of Duke, Lefmann, and R\"odl is that the bounds do not depend on the number of colors. Before stating it, we will need to recall some standard terminology.
\begin{Def}
	Let $\varepsilon>0$ be a parameter, and let $X,Y$ be vertex subsets of a graph $F$. Let $e(X,Y)$ denote the number of pairs in $X \times Y$ that are edges in $F$, and let $d(X,Y)=e(X,Y)/(|X||Y|)$ denote the \emph{edge density} between $X$ and $Y$. We say that the pair $(X,Y)$ is \emph{$\varepsilon$-regular} if for every $X' \subseteq X, Y' \subseteq Y$ with $|X'| \geq \varepsilon |X|, |Y'| \geq \varepsilon |Y|$, we have that
	\[
		|d(X,Y)-d(X',Y')|<\varepsilon.
	\]
	Suppose now that $F$ is $m$-partite, with $m$-partition $V=V_1 \sqcup \dotsb \sqcup V_m$. A \emph{cylinder} $K$ is a set of the form $W_1 \times \dotsb W_m$, where $W_i \subseteq V_i$ for all $i \in [m]$. For such a cylinder $K$, let $V_i(K)=W_i$. We say that $K$ is \emph{$\varepsilon$-regular} if $(W_i,W_j)$ is $\varepsilon$-regular for all $1 \leq i<j \leq m$. A \emph{cylinder partition} $\K$ is a partition of $V_1 \times \dotsb \times V_m$ into cylinders, and we say that $\K$ is \emph{$\varepsilon$-regular} if at most an $\varepsilon$-fraction of the $m$-tuples $(v_1,\ldots,v_m) \in V_1 \times \dotsb \times V_m$ are not in $\varepsilon$-regular cylinders of $\K$. 
\end{Def}
\begin{lem}[Duke--Lefmann--R\"odl \cite{DuLeRo}, Fox--Li {\cite[Theorem 7.2]{FoLi}}]\label{lem:dlr}
	Let $r\geq 1,m \geq 2$ be integers, $0<\varepsilon<\frac 12$ a parameter, and define $\beta=\varepsilon^{m^2 \varepsilon^{-5}}$. Suppose that $F=(V,E)$ is an $m$-partite graph with $m$-partition $V_1 \sqcup \dotsb \sqcup V_m$ whose edges are $r$-colored, so that $E=E(F_1)\sqcup\dotsb \sqcup E(F_r)$. Then there exists a cylinder partition $\K$ of $V_1 \times \dotsb \times V_m$ into at most $4^{m^2 \varepsilon^{-5}}$ cylinders that is $\varepsilon$-regular in each of the graphs $F_1,\ldots,F_r$. Moreover, for each $K \in \K$ and $i \in [m]$, we have that $|V_i(K)| \geq \beta |V_i|$. 
\end{lem}

We will use this result in conjunction with our main technical lemma below. If $V(H)=k$ and $\Gamma$ is a $k$-partite graph with parts $W_1,\ldots,W_k$, we say a copy of $H[t]$ in $\Gamma$ is \emph{canonical} if all copies of vertex $i \in V(H)$ are in part $W_i$ of $\Gamma$. This generalizes the earlier definition of canonical copies in case $\Gamma$ is a blowup $G[n]$. 
\begin{lem}\label{lem:nikiforov-variant}
	Let $H$ be a graph with $V(H)=[k]$, where we suppose that $\{1,2\} \in E(H)$. For every $ij \in E(H)$, let $p_{ij} \in (0,1]$ be a real number, with $p_{12}\leq \frac{1}{2}$, and let $0<\alpha<\prod_{ij \in E(H)}p_{ij}$ be another parameter. Then the following holds for sufficiently large $n$. Suppose that $\Gamma$ is a $k$-partite graph with $k$-partition $W_1 \sqcup \dotsb \sqcup W_k$, with $|W_i| \geq n$ for all $i$. Suppose that whenever $ij \in E(H)$, the pair $(W_i,W_j)$ is $\frac{\alpha^2}{8k^2}$-regular with $d(W_i,W_j) \geq p_{ij}$. Then $\Gamma$ contains a canonical copy of $H[t]$, where
	\[
		t=\left({\prod_{ij \in E(H) \setminus \{1,2\}}}p_{ij}-\alpha\right) \frac{\log n}{\log \frac 1{p_{12}}} .
	\]
\end{lem}
\begin{rem}
	Note that by relabelling the vertices of $H$, we can exclude any $p_{ij}$ we want from the product and instead replace the factor $(\log \frac1{p_{12}})^{-1}$ by $(\log \frac1{p_{ij}})^{-1}$, as long as $p_{ij}\leq \frac{1}{2}$. As $y=(x\log 1/x)^{-1}$ is a decreasing function of $x$ for $x \in [0,1/e]$ and is bounded for $x\in [1/e,1/2]$, this result is strongest, up to an absolute constant factor, when we pick $p_{12}$ to be the minimum of the $p_{ij}$. 
\end{rem}

\begin{proof}[Proof of Lemma \ref{lem:nikiforov-variant}]
In a blowup $H[t]$, we call $t$ vertex-disjoint copies of $H$ a {\it perfect matching of copies of $H$}. 
Let $H_i$ be the subgraph of $H$ induced on the first $i$ vertices, and for $j>i$, let $N_{i}(j)$ denote the set of neighbors $\ell$ of vertex $j$ in graph $H$ with $\ell \leq i$. We let $\textrm{deg}(i)$ denote the degree of vertex $i$ in $H$. We also set $\varepsilon=\alpha^2/(8k^2)$ and $\delta=8 k \varepsilon/(p_{12}\log \frac{1}{p_{12}})$; observe that both $\delta$ and $\varepsilon$ are in $(0,1)$ and do not depend on $n$. Finally, let $q_i=\prod_{\ell \in N_{i-1}(i)}p_{\ell i}$ for $2 \leq i \leq k$, and let $t_1=(1-\deg(1)\varepsilon)|W_1|$, $t_2=(1-\delta)\log n/\log \frac{1}{p_{12}}$, and $t_{i}=\left(q_i-k \varepsilon\right)t_{i-1}$ for $3\leq i\leq k$.

A copy of $H_i$ in $\Gamma$ is {\it canonical} if the copy of vertex $j$ is in $W_j$ for $j \leq i$. A copy of $H_i$ in $\Gamma$ is {\it good} if it is canonical and for each $j>i$, the number of extensions of this copy of $H_i$ to a copy of the induced subgraph $H[\{1,\ldots,i\} \cup \{j\}]$ with the copy of vertex $j$ in $W_j$ is at least $(\prod_{\ell \in N_i(j)}(p_{\ell j}- \varepsilon))|W_j|$.

We prove by induction on $i$ for $1\leq i\leq k$ that we can find a copy of $H_i[t_i]$ which contains a perfect matching $M_i$ of copies of $H_i$, each of which is good. Observe that by regularity, for any $ij\in E(H)$ and subset $W_j'\subseteq W_j$ with $|W_j'|\geq \varepsilon |W_j|$, the number of vertices in $W_i$ with less than $(p_{ij}-\varepsilon)|W_j'|$ neighbors in $W_j'$ is less than $\varepsilon |W_i|$. So, all but at most $(\deg(i)-|N_{i-1}(i)|)\varepsilon |W_i|$ vertices in $W_i$ have degree at least $(p_{ij}-\varepsilon)|W_j'|$ to $W_j'$ for all neighbors $j>i$. In particular, applying this observation with $i=1$ and $W_j'=W_j$ for all $j$ yields that $W_1$ contains at least $t_1=(1-\deg(1)\varepsilon)|W_1|$ good vertices (i.e. good copies of $H_1$), which together trivially form a perfect matching $M_1$. This proves the base case $i=1$ of our induction.

For the inductive step, assume that our claim has been shown for a given $i$. Fix a copy $L_i$ of $H_i$ in the perfect matching $M_i$ of good copies of $H_i$. For $j>i$, let $W_{j,i}$ denote the subset of vertices in $W_j$ which together with $L_i$ form induced copies of $H[\{1,\ldots,i\} \cup \{j\}]$. Since $L_i$ is good, we have $|W_{j,i}| \geq (\prod_{\ell \in N_i(j)}(p_{\ell j}- \varepsilon))|W_j|$ for each $i<j \leq k$. A vertex $v$ in $W_{i+1,i}$ together with $L_i$ form a good copy of $H_{i+1}$ so long as $v$ has degree at least $(p_{(i+1)j}-\varepsilon)|W_{j,i}|$ to $W_{j,i}$ for each neighbor $j>i+1$ of $i+1$. Applying the regularity observation above with $W_j'=W_{j,i}$, we conclude that the number of such $v$ is at least 
\begin{align*}|W_{i+1,i}|-&(\deg(i+1)-|N_i(i+1)|)\varepsilon|W_{i+1}| \\
&\geq  \left(\left(\prod_{\ell \in N_{i}(i+1)}(p_{\ell (i+1)}-\varepsilon)\right)-(\deg(i+1)-|N_i(i+1)|)\varepsilon\right)|W_{i+1}| \\ 
& \geq  \left(\prod_{\ell \in N_i(i+1)}p_{\ell (i+1)}-\deg(i+1)\varepsilon\right)|W_{i+1}| \\ & \geq  \left(q_{i+1}-k\varepsilon\right)|W_{i+1}|  .\end{align*}

Consider the auxiliary bipartite graph $B$ with parts $M_i$ and $W_{i+1}$, where a copy $L_i$ of $H_i$ in $M_i$ and a vertex $w \in W_{i+1}$ are adjacent if $L_i$ together with $w$ form a good copy of $H_{i+1}$.  In $B$, each vertex in $M_i$ has degree at least $\left(q_{i+1}-k \varepsilon\right)|W_{i+1}|$, and hence $B$ has edge density at least $\rho:=q_{i+1}-k \varepsilon$. For the rest of the argument, we split into two cases to deal with the smallest case separately:

\emph{Case 1: $i+1=2$}. In this case, $M_1$ is actually a subset of $W_1$. By adding back in the remaining vertices of $W_1$ as disconnected vertices, we can view $B$ as a bipartite subgraph of $\Gamma$ between $W_1$ and $W_2$, with edge density at least $\rho \frac{|M_1|}{|W_1|}=(1-\deg(1)\varepsilon)(p_{12}-k\varepsilon)\geq p_{12}-2k\varepsilon$. Then, by deleting vertices of lowest degree from each part one at a time, we can find an induced subgraph with exactly $n$ vertices in each part and edge density at least $p_{12}-2k \varepsilon$ between its parts. The K\H ovari--S\'os--Tur\'an theorem \cite{KoSoTu} implies that a $K_{r,r}$-free bipartite graph where both parts have $n$ vertices has at most $(r-1)^{1/r}n^{2-1/r}+(r-1)n$ edges. Let $r=t_2=(1- \delta)\log_{1/p_{12}}n$. Observe that
\[
	\left( \frac rn \right) ^{1/r} \leq \left( \frac{\log_{\frac1{p_{12}}} n}{n} \right)^{1/r} = \exp \left[ \left(\log \frac 1{p_{12}}\right) \left( \frac{\log \log n-\log \log (1/p_{12})}{(1- \delta)\log n}- \frac{1}{1- \delta} \right)  \right] \leq p_{12}^{1+3 \delta/4}
\]
for $n$ sufficiently large in terms of $p_{12}$. Also for $n$ sufficiently large, we have that $r/n \leq k \varepsilon$. By the definition of $\delta$, we see that $p_{12}^{3 \delta/4} = e^{-6k \varepsilon/p_{12}} \leq 1-3k \varepsilon/p_{12}$, using the inequality $e^{-x} < 2^{-x} \leq 1-x/2$ for $x \in [0,1]$. Therefore, we find that
\begin{align*}
	(r-1)^{1/r}n^{2-1/r}+(r-1)n&< \left( \left( \frac rn \right) ^{1/r}+\frac rn \right) n^2 \leq \left( p_{12}^{1+3 \delta/4}+k \varepsilon \right) n^2 \leq (p_{12}-2k \varepsilon)n^2.
\end{align*}
Thus, $B$ contains a $K_{r,r}$, since it has a bipartite subgraph with $n$ vertices in each part and at least $(p_{12}-2k \varepsilon)n^2$ edges. This $K_{r,r}$ corresponds to a canonical $H_2[t_2]$ in $\Gamma$, all of whose edges are good; we finish by choosing any perfect matching $M_2$ inside this $H_2[{t_2}]$.

\emph{Case 2: $i+1 > 2$}. In this case, the average degree of vertices in $W_{i+1}$ in $B$ is at least $\rho t_i = t_{i+1}$. For a given vertex $w \in W_{i+1}$, letting $\textrm{deg}_B(w)$ denote the degree of $w$ in graph $B$, there are exactly $\binom{\textrm{deg}_B(w)} {t_{i+1}}$ pairs $(w,S)$ with $w \in W_{i+1}$ and $S$ is a subset of $M_i$ of size $t_{i+1}$ and in $B$ the vertex $w$ is adjacent to all vertices in $S$. So the total number of such pairs, ranging over all vertices $w \in W_{i+1}$, is $\sum_{w \in W_{i+1}} \binom{\textrm{deg}_B(w)}{t_{i+1}}$. Define the convex function $f$ by \[f(x)=\begin{cases}\binom x{t_{i+1}} & \text{if }x \geq t_{i+1}-1\\0 & \text{if } x < t_{i+1}-1\end{cases},\] which agrees with $\binom x{t_{i+1}}$ if $x$ is a nonnegative integer. Applying Jensen's inequality to $f$, we see that there are at least $n$ pairs $(w,S)$, where $S$ is a subset of $M_i$ of size $t_{i+1}$ and in $B$ the vertex $w$ is adjacent to all vertices in $S$. The number of subsets $S$ of $M_i$ of size $t_{i+1}$ is $\binom{t_i}{t_{i+1}} \leq 2^{t_i} \leq 2^{t_2}\leq n^{1- \delta}$, so there is such a set $S \subset M_i$ for which at least $n/n^{1- \delta}=n^{\delta} \geq t_{i+1}$ vertices $w$ are adjacent to all vertices in $S$ in the bipartite graph $B$, as long as $n$ is large enough so that $n^\delta \geq \log_{1/p_{12}} n \geq t_{i+1}$. These $t_{i+1}$ copies of $H_i$ together with $t_{i+1}$ such vertices $w \in W_{i+1}$ form the vertex set of a copy of $H_{i+1}[t_{i+1}]$ which has a matching $M_{i+1}$ of good copies of $H_{i+1}$  which extends the matching $M_i$ of good copies of $H_i$. Thus in either case we get a copy of $H_{i+1}[t_{i+1}]$ with the desired properties. This completes the induction proof.

Hence, we get a copy of $H_{k}[t_k]=H[t]$ with 
	\begin{align*}
		t_k&=\left(q_k-k \varepsilon\right)\left(q_{k-1}-k \varepsilon\right) \dotsb \left(q_3-k \varepsilon\right)t_2 \\
		&\geq \left(\prod_{ij \in E(H) \setminus \{1,2\}}p_{ij}-k(k-2) \varepsilon\right)\left(1- \delta\right) \frac{\log n}{\log \frac{1}{p_{12}}} \\
		&\geq \left({\prod_{ij \in E(H) \setminus \{1,2\}}}p_{ij}-k^2 \varepsilon- \delta\right) \frac{\log n}{\log \frac 1{p_{12}}} \\
		&\geq\left({\prod_{ij \in E(H) \setminus \{1,2\}}}p_{ij}-\alpha\right) \frac{\log n}{\log \frac 1{p_{12}}} ,
	\end{align*}
where the last step uses that $k^2 \varepsilon=\alpha^2/8 \leq \alpha/6$ and that
\[
	\delta=\frac{8 k \varepsilon}{p_{12}\log \frac{1}{p_{12}}}=\alpha \cdot \frac{\alpha}{k p_{12}\log \frac{1}{p_{12}}} \leq \alpha\cdot \frac{p_{12}}{k p_{12}\log \frac{1}{p_{12}}} \leq \frac {5\alpha}6,
\]
since $\alpha < \prod p_{ij}\leq p_{12}$ and $k\log \frac{1}{p_{12}}\geq 2\log 2 > \frac{6}{5}$. This is precisely the blowup we were looking for.
\end{proof}

Finally, we will need a standard counting lemma in Section \ref{sec:nikiforov}.
\begin{lem}[See e.g.\ {\cite[Theorem 3.30]{Zhao}}]\label{thm:counting}
	Let $H$ be a graph with $V(H)=[k]$, and let $\Gamma$ be a graph with disjoint vertex subsets $W_1,\ldots,W_k$. Suppose that $(W_i,W_j)$ is $\varepsilon$-regular for all $ij \in E(H)$. Let $N(H)$ denote the number of canonical copies of $H$ in $\Gamma$, i.e.\
	\[
		N(H)=|\{(w_1,\ldots,w_k) \in W_1 \times \dotsb \times W_k: w_i w_j \in E(\Gamma) \text{ for all }ij \in E(H)\}|.
	\]
	Then
	\[
		\left\lvert N(H)-\prod_{ij \in E(H)} d(W_i,W_j) \cdot \prod_{i=1}^k |W_i|\right\rvert \leq \varepsilon|E(H)|  \prod_{i=1}^k |W_i|.
	\]
\end{lem}
\begin{rem}
	Usually, the counting lemma is stated for the number of homomorphisms from $H$ to $\Gamma$, which might be larger by a lower-order term than the number of copies of $H$. However, since we require $W_1,\ldots,W_k$ to be disjoint, these quantities actually coincide. 
\end{rem}

\section{A new proof of Nikiforov's theorem}\label{sec:nikiforov}
Using Lemma \ref{lem:nikiforov-variant}, we can prove a version of Theorem \ref{thm:nikiforov-original} with stronger quantitative dependence in its parameters. Specifically, in this section, we prove the following theorem.
\begin{thm}
If $0<\eta<e^{-1}$, $H$ is a graph on $k$ vertices, and $\lambda=\frac{\eta^{1-1/|E(H)|}}{5\log \frac 1 \eta}$, then the following holds for all $n$ sufficiently large. If $G$ is a graph on $n$ vertices containing at least $\eta n^k$ labeled copies of $H$, then $G$ contains a blowup $H[t]$, where $t=\lambda \log n$. 
\end{thm}
\begin{proof}
	Let $V(H)=[k]$. Consider an equitable partition of $V(G)$ picked uniformly at random with parts $V_1,\ldots,V_k$, each of size $n/k$. Every labeled copy of $H$ has a probability at least $n^{-k}\prod_{i=1}^k |V_i|$ of being canonical with respect to this partition, namely having vertex $i$ in $V_i$ for all $i \in [k]$. Therefore, by linearity of expectation, there exists a partition $V_1,\ldots,V_k$ with $|V_i| =n/k$ for all $i$ and such that $V_1,\ldots,V_k$ contain at least $\eta\prod_{i=1}^k |V_i|$ canonical copies of $H$.

	Let $F$ be the $k$-partite subgraph of $G$ whose parts are $V_1,\ldots,V_k$ obtained by deleting all edges contained in each $V_i$. We apply Lemma \ref{lem:dlr} to $F$, with $m=k$, $r=1$, and $\varepsilon=\eta^{2k^2}/(8k^2)$. We obtain an $\varepsilon$-regular cylinder partition $\K$ of $V_1 \times \dotsb \times V_k$ with $|V_i(K)| \geq \beta n/k$ for all $i$, where $\beta=\varepsilon^{k^2 \varepsilon^{-5}}$. Notice that if $K \in \K$ is an $\varepsilon$-regular cylinder, then the counting lemma implies that the number of canonical copies of $H$ in $K$ is at most
	\[
		\left(\prod_{ij \in E(H)} d(V_i(K),V_j(K)) + \varepsilon |E(H)|\right)  \prod_{i=1}^k |V_i(K)|.
	\]
	Moreover, recall that at most an $\varepsilon$-fraction of the tuples in $V_1 \times \dotsb \times V_k$ are in non-$\varepsilon$-regular cylinders, and in particular at most $\varepsilon \prod_{i=1}^k |V_i|$ canonical copies of $H$ are in such cylinders. Adding these two facts up over all cylinders in $\K$, we find that the total number of canonical copies of $H$ in $F$ is at most
	\[
		\varepsilon \prod_{i=1}^k |V_i|+ \sum_{K\text{ regular}}\left(\prod_{ij \in E(H)} d(V_i(K),V_j(K)) + \varepsilon |E(H)|\right)  \prod_{i=1}^k |V_i(K)|.
	\]
	On the other hand, we know that the number of such copies is at least $\eta\prod_{i=1}^k |V_i|$. Therefore, there must exist an $\varepsilon$-regular cylinder $K$ in the cylinder partition for which
	\[
		\prod_{ij \in E(H)} d(V_i(K),V_j(K))+\varepsilon |E(H)| \geq \eta -\varepsilon.
	\]
	Fixing such a $K$, let $W_i=V_i(K)$, and let $\Gamma$ be the subgraph of $G$ induced on $W_1 \cup \dotsb \cup W_k$. We know that each part of $\Gamma$ has size at least $\beta n/k$. Suppose without loss of generality that $d(W_1,W_2)$ is minimum among all $d(W_i,W_j)$, and let $p_{12}=\min(d(W_1,W_2),\frac{1}{2})$ and $p_{ij}=d(W_i,W_j)$ for all other $ij\in E(H)$. Then by Lemma \ref{lem:nikiforov-variant} (assuming $n$, and thus $\beta n/k$, is sufficiently large), we find that $\Gamma$ contains a copy of $H[t]$, where 
	\[
		t=\left({\prod_{ij \in E(H) \setminus \{1,2\}}}p_{ij}-\alpha\right) \frac{\log (\beta n / k)}{\log \frac 1{p_{12}}},
	\]
	and $\alpha=\sqrt{8\varepsilon k^2}=\eta^{k^2}$. Let $P=\prod_{ij\in E(H)} p_{ij}$. We have $P\geq \frac{1}{2} \prod_{ij\in E(H)}d(V_i(K),V_j(K)) \geq \frac{1}{2} (\eta - (|E(H)|+1)\varepsilon) > \frac{9}{20}\eta > 9 \alpha$ and $P\leq p_{12}\leq P^{1/|E(H)|}$, so for $n$ sufficiently large in terms of $\eta$, we can bound
	\begin{align*}
        t &\geq \log (\beta n / k) \frac{P-p_{12}\alpha}{p_{12}\log \frac{1}{p_{12}}} \geq \log (\beta n / k) \frac{P - \alpha}{P^{1/|E(H)|}\log \frac{1}{P}} \\
        & \geq (\log n + \log (\beta / k))\frac{(P-\frac{1}{9}P) P^{-1/|E(H)|}}{\log \frac{1}{P}} \\
        & \geq \left(\frac{9}{10}\log n\right)\frac{8}{9}\frac{ (\frac{9}{20}\eta)^{1-1/|E(H)|}}{\log \frac{1}{\eta}+\log \frac{20}{9}}
        \geq \log n \frac{9}{10}\frac{8}{9}\frac{9}{20}\frac{ \eta^{1-1/|E(H)|}}{\frac{9}{5}\log \frac{1}{\eta}} \\
        & \geq \frac{\eta^{1-1/|E(H)|}}{5\log \frac 1 \eta} \log n,
	\end{align*}
	as claimed.
\end{proof}
\begin{rem}
	In contrast with Nikiforov's result, where he assumes a bound on the number of unlabeled copies, we work here with labeled copies, which allows us to pick an $\eta$ which is a factor the number of automorphisms of $H$ larger.  

Also, just as in Nikiforov's original proof of Theorem \ref{thm:nikiforov-original}, we can use the same technique to find an unbalanced blowup of $H$. Namely, for any $c>0$, there is a $0<\lambda'<\lambda$ such that we can find a blowup of $H$ in $G$ where the first $k-1$ parts have size $\lambda' \log n$ and the last part has size $n^{1-c}$. Indeed, this follows directly by examining the proof of Lemma \ref{lem:nikiforov-variant}, which shows that at each step, we can actually pick out $n^{1-c}$ vertices in the second part of the auxiliary graph, as long as $t_i$ is decreased by a sufficiently large constant factor.
\end{rem}

\section{Proof of Theorem \ref{thm:no-g-ub}}\label{sec:upper-bound}
In this section, we prove Theorem \ref{thm:no-g-ub}.
\begin{proof}[Proof of Theorem \ref{thm:no-g-ub}]
	Fix $0<\alpha<r^{-|E(H)|}$ and let $\gamma=\alpha 2r^2$ be the parameter in the theorem statement. Let $m=|V(G)|$ and $k=|V(H)|$. Let $a=a(G,H,r)$ and $b=b(H,r)$ be parameters to be defined later, and let $n=a \cdot b^t$. Let $F=G[n]$, and fix an $r$-coloring $E(F)=E(F_1)\sqcup \dotsb \sqcup E(F_r)$; we wish to show that this coloring contains a monochromatic canonical copy of $H[t]$. We identify $V(G)$ with $[m]$, and let $V_1,\ldots,V_m$ be the parts of $F=G[n]$. We also identify the vertex set of $H$ with $[k]$. 

	Let $\varepsilon=\alpha^2/(8k^2)$ be the parameter from Lemma \ref{lem:nikiforov-variant}. We apply Lemma \ref{lem:dlr} with parameters $r,m$ and $\varepsilon$. Then we obtain a cylinder partition $\K$ of $V_1 \times \dotsb \times V_m$ which is $\varepsilon$-regular for each of the color classes $F_1,\ldots,F_r$. Fix an $\varepsilon$-regular cylinder $K \in \K$, say $K=W_1 \times \dotsb \times W_m$. By Lemma \ref{lem:dlr}, we have $|W_i| \geq \beta n$ for all $i \in [k]$, where $\beta=\varepsilon^{m^2 \varepsilon^{-5}}$. Define an $r$-coloring of $E(G)$ by coloring the edge $ij$ by the most popular color in $W_i \times W_j$, breaking ties arbitrarily. Since $G \rarrow H$, this $r$-coloring must contain a monochromatic copy of $H$. By renaming the colors and the parts, we may assume without loss of generality that this copy of $H$ is on the vertices $1,\ldots,k$, so that $ij$ is of color $1$ if $ij$ is an edge of $H$.

	Therefore, we find that among all the pairs $(W_i,W_j)$ where $1 \leq i<j \leq k$ and $ij$ is an edge of $H$, we have that color $1$ is the densest color in $(W_i,W_j)$. Let $\Gamma$ be the induced subgraph of $F_1$ on $W_1 \cup \dotsb \cup W_k$. Then we know that each pair $(W_i,W_j)$ with $ij$ an edge of $H$ is $\varepsilon$-regular in $\Gamma$ (since the cylinder $K$ was $\varepsilon$-regular in each color) and satisfies $d_\Gamma(W_i,W_j) \geq p$, where $d_\Gamma$ denotes the edge density in $\Gamma$, and $p=1/r$. Since $\alpha< r^{-|E(H)|}=p^{|E(H)|}$, we may apply Lemma \ref{lem:nikiforov-variant} with all $p_{ij}$ equal to $p$ to find a canonical blowup $H[t^*]$ (which is monochromatic), where
	\[
		t^* = \frac{p^{|E(H)|-1}-\alpha}{\log \frac 1p}  \log (\beta n)= \frac{r^{1-|E(H)|}-\alpha}{\log r}  \log (\beta n).
	\]
	Now, we define $a=a(G,H,r)=1/\beta$ and $b=b(H,r)=r^{r^{|E(H)|-1}(1+\alpha r^{|E(H)|})}$, so that
	\begin{align*}
		t^*&=\frac{r^{1-|E(H)|}-\alpha}{\log r} \log (b^t) = t \left[ (r^{1-|E(H)|}-\alpha)(r^{|E(H)|-1}(1+\alpha r^{|E(H)|})) \right] \\
		& =t \left[ (1- \alpha r^{|E(H)|-1})(1+\alpha r^{|E(H)|}) \right] = t\left[ 1 + \alpha r^{|E(H)|}(1 - r^{-1} - \alpha r^{|E(H)|-1}) \right] \\
		&\geq t\left[1 + \alpha r^{|E(H)|}(1 - 2r^{-1})\right] \\
		& \geq t.\qedhere
	\end{align*}
\end{proof}

\section{Concluding remarks}
In addition to eliminating the unnecessary dependence on $G$ in the exponential constant of $\B(G \rarrow H;t)$, Theorem \ref{thm:no-g-ub} also provides quite good bounds on the exponential constant in many instances. For instance, Souza's results \cite{Souza} imply the bounds 
\[
	2^t \leq \B(K_6 \rarrow[2] K_3;t) \leq e^{(3.3 \times 10^7)t},
\]
and he asked whether the upper bound could be made more reasonable. Theorem \ref{thm:no-g-ub} implies
\[
	\B(K_6 \rarrow[2] K_3;t) \leq 2^{(4+o(1))t}=(16+o(1))^t.
\]
Moreover, the same bound holds for $\B(G \rarrow[2] K_3;t)$ for any graph $G$ with $G \rarrow[2] K_3$, as long as the $o(1)$ term above is allowed to depend on $G$. We expect that the upper bound can be improved further using some of our techniques, but such an improvement would likely require some new ideas.

The most natural question left open by Theorem \ref{thm:no-g-ub} is whether the dependence on $G$ can be entirely eliminated, or whether $\B(G \rarrow H;t)$ must depend on $G$. Unlike Souza, we believe the latter to be the case, and make the following conjecture.
\begin{conj}\label{conj:lb-with-g}
	There exists a graph $H$ and integers $r,t \geq 2$ for which the following holds. There exist graphs $G_1,G_2,\ldots$ such that $G_i \rarrow H$ for all $i$ and $\sup_i \B(G_i \rarrow H;t)=\infty$. 
\end{conj}
\noindent We even conjecture this holds with $H$ is a triangle and $r=t=2$.   
\begin{conj}
	For every $s$, there exists a graph $G$ such that $G \rarrow[2] K_3$ but $G[s] \nrarrow[2] K_3[2]$. 
\end{conj}
For certain graphs $H$ and integers $r$, Conjecture \ref{conj:lb-with-g} does not hold, and $\B(G \rarrow H;t)$ can be bounded by an exponential function independent of $G$. One example of such graphs, as observed by Souza, are the $r$-Ramsey-finite graphs. Let $\M_r(H)$ denote the set of all $G$ which are minimal with respect to the property $G \rarrow H$, i.e.\ all graphs $G$ with $G \rarrow H$ but $G' \nrarrow H$ for any proper subgraph $G'$ of $G$. $H$ is called \emph{$r$-Ramsey-finite} if $|\M_r(H)|<\infty$, and $r$-Ramsey-infinite otherwise. If $H$ is $r$-Ramsey-finite, then $\B(G \rarrow H;t) \leq c^t$ for a constant $c$ that does not depend on $G$; indeed, we may find such a $c$ by taking the maximum $c$ from Theorem \ref{thm:souza-ub} over all $G \in \M_r(H)$. 

However, there is at least one Ramsey-infinite graph $H$ (namely the path $P_3$ with two edges) for which Conjecture \ref{conj:lb-with-g} fails to hold and further $\B(G \rarrow[2] H;t) \leq c^t$ for all $G$ with $G \rarrow[2] H$ where $c$ does not depend on $G$. Indeed, $\M_2(P_3)$ is infinite, consisting of $K_{1,3}$ and the odd cycles. Equivalently, $G\rarrow[2] P_3$ if and only if $G$ has a vertex of degree at least $3$ or $G$ contains an odd cycle $C_{2\ell+1}$. If $G$ has a vertex of degree at least $3$, then $\B(G \rarrow[2] P_3;t) \leq \B(K_{1,3} \rarrow [2] P_3;t)$, so we can use the same upper bound for all such $G$. On the other hand, it is a simple exercise to show that for each $\varepsilon>0$ there is $\delta>0$ such that if a $2$-edge-coloring of $P_4[n]$ has at most $\delta n^3$ monochromatic canonical $P_3$, then, apart from at most $\varepsilon n^2$ edges, the coloring is monochromatic between consecutive parts and alternates color along the path. In particular, taking $\varepsilon = 1/3$, if a $2$-edge-coloring of $C_{2\ell+1}[n]$ does not contain $\frac{\delta}{2}n^3$ monochromatic canonical $P_3$ between any three consecutive parts, then the most common color used between consecutive pairs of parts alternates along the cycle, contradicting that an odd cycle is nonbipartite. That is, every $2$-edge-coloring of $C_{2\ell+1}[n]$ must contain at least $\frac{\delta}{2}n^3$  monochromatic canonical $P_3$ between some three consecutive parts. Applying Nikiforov's theorem between these three consecutive parts, there is a monochromatic canonical copy of $P_3[t]$ with $t = \Omega(\log n)$ and the implicit constant is absolute.  Hence, although $P_3$ is not $2$-Ramsey-finite, there is still an absolute constant $c$ such that $\B(G \rarrow[2] P_3;t) \leq c^t$ for all $G$ with $G \rarrow[2] P_3$. 

Souza defined the \emph{robustness} $\beta_r(H;G)$ to be the minimum number of monochromatic copies of $H$ in an $r$-coloring of $G$, divided by the total number of copies of $H$ in $G$. Thus, $\beta_r(H;G)$ measures the fraction of copies of $H$ that must be monochromatic in any $r$-coloring of $G$. He also showed, again as a consequence of Theorem \ref{thm:souza-ub}, that if $\inf \{\beta_r(H;G):G \in \M_r(H)\}>0$, then Conjecture \ref{conj:lb-with-g} fails to hold for $H$. If $H$ is $r$-Ramsey-finite, then this infimum is certainly positive, so we recover the above observation that Conjecture \ref{conj:lb-with-g} fails for $r$-Ramsey-finite graphs. Moreover, Souza \cite[Conjecture~5.4]{Souza} conjectured that these two observations are in fact the same, namely that $\inf \{\beta_r(H;G):G \in \M_r(H)\}>0$ if and only if $H$ is $r$-Ramsey-finite. Indeed, this conjecture is true.
\begin{prop}
	If $H$ is $r$-Ramsey-infinite, then $\inf \{\beta_r(H;G):G \in \M_r(H)\}=0$.
\end{prop}

This proposition follows from the next lemma and $\sup \{e(G):G \in \M_r(H)\}=\infty$ if $H$ is $r$-Ramsey-infinite; this fact follows from the observation that a Ramsey-minimal graph for $H$ can have at most as many isolated vertices as $H$ itself, so the number of edges of $G$ must tend to infinity as $G$ runs over the infinite set $\M_r(H)$. 

\begin{lem}
If $G$ is Ramsey minimal for $H$ with $r$ colors, then $\beta_r(H;G) \leq \frac{e(H)}{re(G)}$. 
\end{lem}
\begin{proof}
	If we fix a copy of $H$ in $G$ and then pick an edge of $G$ uniformly at random, the probability that it lands in this copy is exactly $e(H)/e(G)$. Therefore, by linearity of expectation, there exists an edge $e \in E(G)$ such that $e$ lies in at most an $e(H)/e(G)$ fraction of the copies of $H$ in $G$. Since $G$ is Ramsey-minimal for $H$, we can color $G-e$ so that it contains no monochromatic copy of $H$. We then color $e$ according to which color it would participate in the least number of monochromatic copies of $H$. 
We thus find that the total fraction of copies of $H$ that are monochromatic is at most $\frac{e(H)}{re(G)}$, since every such copy must contain $e$ and there are $r$ colors. Thus, $\beta_r(H;G) \leq \frac{e(H)}{re(G)}$. 
\end{proof}

It is natural to modify the definition of blowup Ramsey numbers to allow for non-canonical copies. More precisely, we can define
\[
	\B'(G,H,r,t)=\min \{n:G[n] \rarrow H[t]\}.
\]
Note that $\B(G\rarrow H;t)$ is finite if and only if $G \rarrow H$; the if direction was proven by Souza, while the only if follows from blowing up any coloring of $G$ with no monochromatic copy of $H$. However, $\B'(G,H,r,t)$ can be finite for all $t$ even if $G\nrarrow H$. Indeed, let's say that $G \homrarrow H$ if every $r$-edge-coloring of $G$ contains a monochromatic homomorphic image of $H$, where we say that $H'$ is a homomorphic image of $H$ if it can be gotten from $H$ by repeatedly identifying non-adjacent vertices. In this case, a sufficiently large blowup of $H'$ will contain a copy of $H$. Therefore we can conclude that $\B'(G,H,r,t)$ is finite if and only if $G\homrarrow H$, where the only if direction follows by blowing up a coloring of $G$ containing no monochromatic homomorphic image of $H$. Moreover, we thus find that 
\[
	\B'(G,H,r,t) \leq \B(G \rarrow H;t) \leq \B'(G,H,r,ct),
\]
where $c=c(H) \geq 1$ is a constant depending on how small a homomorphic image of $H$ can be. Thus, $\frac 1t \log \B(G \rarrow H;t)$ and $\frac 1t \log \B'(G,H,r,t)$ differ only by a constant factor depending on $H$.

\end{document}